\documentclass[12pt]{article}

\usepackage{amsmath}
\usepackage{amssymb}
\usepackage{amsthm}
\usepackage{amsxtra}
\usepackage{algorithm}
\usepackage{bm}
\usepackage{hyperref}
\usepackage[all]{xy}
\usepackage{yhmath}
\usepackage{epsfig,graphicx,psfrag,mathrsfs,subfigure}
\usepackage[noend]{algpseudocode}
\usepackage{times}
\usepackage{setspace}

\newtheorem{Def}{Definition}[section]
\newtheorem{Theo}[Def]{Theorem}
\newtheorem{Pro}[Def]{Proposition}
\newtheorem{Cor}[Def]{Corollary}
\newtheorem{Rem}[Def]{Remark}
\newtheorem{Lem}[Def]{Lemma}

\begin{document}
\begin{spacing}{1.5}
\noindent\rule[0.25\baselineskip]{\textwidth}{1pt}

\noindent \textbf{\Large{The relation between Parabolic Hecke modules and $W$-graph ideal modules in Kazhdan-Lusztig theory}}\newline\newline
\noindent \textbf{Qi Wang} \newline
\end{spacing}
\begin{spacing}{1.2}
\noindent\textbf{Abstract }
In 2011, Howlett and Nguyen \cite{r1} introduced the concept of a $W$-graph ideal $E_J$ in $\left ( W,\leqslant _{L} \right )$ with respect to $J$ (a subset of $S$), where $\leqslant _{L}$ is the left weak order on $W$. They proved that one can construct a $W$-graph from a given $W$-graph ideal by constructing a Hecke module structure on $E_J$, where the $W$-graph was introduced by Kazhdan and Lusztig in \cite{d1}.

In this paper, we give the relation between Hecke modules on $E_J$ and general Hecke algebras by considering the relation between Hecke modules on $E_J$ and parabolic Hecke modules. And inspired by Lusztig \cite{g3}, we show that the parabolic Hecke module is isomorphic to a left ideal of the Hecke algebra. Lastly, we give the relation between $R$-polynomials on $E_J$ and parabolic $R$-polynomials as an application of the main results.\newline

\noindent\textbf{Keywords:} Coxeter group, Hecke algebra, Kazhdan-Lusztig theory, $W$-graph, $W$-graph ideal,  Parabolic Hecke modules.\newline\newline\newline

\section*{Introduction}
Let $(W,S)$ be a Coxeter system and $\mathscr{H}\left ( W,S \right )$ be the corresponding Hecke algebra. There is a representation of $\mathscr{H}\left ( W,S \right )$ called $W$-graph that introduced by Kazhdan and Lusztig in \cite{d1}. A $W$-graph provides a method for constructing a matrix representation of $\mathscr{H}\left ( W,S \right )$, the element in this matrix is the so called Kazhdan-Lusztig polynomials and the degree of the representation being the number of vertices of the $W$-graph.

In \cite{v1} Deodhar used parabolic Kazhdan-Lusztig polynomials relative to a standard parabolic subgroup $W_J$ ($J$ is a subset of $S$) to give $W$-graph structures on $D_J$, where $D_J$ is the set of minimal coset representatives of $W_J$. In \cite{h1} Tagawa introduced the weighted case for Deodhar's constructions.

In \cite{r1} Howlett and Nguyen introduced the concept of a $W$-graph ideal in $(W,\leqslant_L)$ with respect to a subset $J$ (of $S$), where $\leqslant_L$ is the left weak Bruhat order on $W$. Then they showed that a $W$-graph can be constructed from a given $W$-graph ideal. In \cite{y1} and \cite{y2} Yin introduced the weighted case for Howlett's constructions.

In this paper, we give the relations between Howlett's constructions and constructions of previous scholars. In order to generality, we give the weighted case.

The paper is organised as follows. In Section 1, we present some basic concepts and recall the definitions of parabolic Hecke modules and $W$-graph ideals with some modifications. In Section 2, we construct the relation between parabolic Hecke modules on $D_J$ and $W$-graph ideal modules in order to give the relation between $W$-graph ideal modules and general Hecke algebras. In Section 3, we show that the parabolic Hecke module is isomorphic to a left ideal of the Hecke algebra. In Section 4, we give the relations between parabolic $R$-polynomials and $R$-polynomials on $E_J$.

\section{Preliminaries}

In this section we follow the conventions in Yin \cite{y1}. Let $\Gamma$ be a totally ordered Abel group with the zero element $\bm{0}$ which will be denoted additively, the order on $\Gamma$ will be denoted by $\leqslant $.

Let $\mathbb{Z}[\Gamma ]$ be a free $\mathbb{Z}$-module with basis $\left \{ q^{\gamma }\mid \gamma \in \Gamma  \right \}$, where $q$ is an indeterminant. A map $L:W\rightarrow \Gamma $ is called a weight function of $W$ into $\Gamma$ if $L$ satisfies that
\begin{center}
$L(w)=L\left (s_{1} \right )+L\left (s_{2}  \right )+\cdots +L\left (s_{k}  \right )$,
\end{center}
for any reduced expression $w=s_{1}s_{2}\cdots s_{k}$. We assume throughout that $L(s)\geqslant \bm{0}$ for all $s\in S$.

Let $\mathscr{H}=\mathscr{H}\left ( W,S,L \right )$ be the weighted Hecke algebra corresponding to $(W,S)$ with parameters $\left \{ q_s^{1/2} \mid s\in S\right \}$, where $q_s=q^{L(s)}$. It has a basis set $\left \{ T_{w}\mid w\in W \right \}$ as a free $\mathbb{Z}[\Gamma ]$-module. The multiplication is given by the rules
  \begin{center}
    $T_{s}T_{w}=
    \left\{\begin{aligned}
      &T_{sw}                            &\text{if}\ \ell(sw)>\ell(w)\\
      &q_{s}T_{sw}+(q_{s}-1)T_{w}  &\text{if}\ \ell(sw)<\ell(w).
    \end{aligned}\right.$
  \end{center}
Let $\bar{\ \ }: \mathscr{H}\rightarrow \mathscr{H}$ be the unique ring involution such that $\overline{q_{w}T_{w}}=q_w^{-1}T_{w^{-1}}^{-1}$ for any $w\in W$.

\subsection{The parabolic Hecke modules}
Let $W_J=\left \langle J \right \rangle$ be the parabolic subgroup of $W$, denote
\begin{center}
$D_J=\left \{ w\in W\mid \ell (ws)> \ell (w)\, \text{for all} \ s\in J\right \}$,
\end{center}
be the set of minimal coset representatives of $W_J$. It has been proved that there are two parabolic Hecke modules $M^{J}$ and $\widetilde{M}^{J}$ on $D_J$ by Deodhar \cite{v1}, and the duality of this two modules was introduced in Deodhar \cite{v2}. The weight case was introduced by Tagawa in \cite{h1}. We recall these constructions as follows.

For each $\sigma\in D_{J}$ define the following subsets of $S$:
   \begin{center}
   $\begin{aligned}
  D_{J, \sigma}^{-}&=\left \{s\in S \mid \ell(s\sigma)<\ell (\sigma)\right \},\\
  D_{J, \sigma}^{+}&=\left \{s\in S \mid \ell(s\sigma)>\ell (\sigma)\ \text{and}\ s\sigma\in D_{J} \right \},\\
  D_{J, \sigma}^{0}&=\left \{s\in S \mid \ell(s\sigma)>\ell (\sigma)\ \text{and}\ s\sigma\notin D_{J} \right \}.
    \end{aligned}$
  \end{center}

Let $u_s$ be a root of the equation $x^2=q_s+(q_s-1)x$. i.e., $u_s=-1$ or $u_s=q_s$. There exists a $\mathscr{H}$-module $M^{J,u_s}$ which is free as a $\mathbb{Z}[\Gamma ]$-module with a basis $\left \{ m_{\sigma}^{J,u_s}\mid \sigma\in D_J \right \}$ such that the actions of $\mathscr{H}$ is given by
\begin{center}
    $T_{s}\cdot m_{\sigma}^{J,u_s}=
    \left\{\begin{aligned}
      &q_sm_{s\sigma}^{J,u_s}+(q-1)m_{\sigma}^{J,u_s}     &\text{Èô}\ s\in D_{J, \sigma}^{-},\\
      &m_{s\sigma}^{J,u_s}                   &\text{Èô}\ s\in D_{J, \sigma}^{+},\\
      &u_sm_{\sigma}^{J,u_s}                   &\text{Èô}\ s\in D_{J, \sigma}^{0}.
    \end{aligned}\right.$
\end{center}
The weighted parabolic Hecke module $M^{J,u_s}$ has an involution $\bar{ \ }$ which is compatible with the involution on $\mathscr{H}$, i.e., for $h\in\mathscr{H}$, $m \in M^{J,u_s}$, $\overline{h\cdot m}=\overline{h}\cdot \overline{m}$.

In this paper, the Hecke module $M^{J,-1}$ is denoted by $M^{J}$ and the Hecke module $M^{J,q_s}$ is denoted by $\widetilde{M}^{J}$.

There exists an algebra map $\Phi : \mathscr{H}\rightarrow \mathscr{H} $ given by
\begin{center}
$\Phi (q_s^{1/2})=-q_s^{1/2}$ and $\Phi (T_w)=\epsilon _{w}q_{w}\overline{T_w}$,
\end{center}
where $\bar{\ }$ is the standard involution in $\mathscr{H}$ and $\epsilon _{w}=(-1)^{\ell(w)}$. Furthermore, $\Phi^2=Id$ and $\Phi$ commutes with $\bar{\ }$.

There is a map $\varphi _J$ (respectively $\widetilde{\varphi} _J$) from $\mathscr{H}$ to $M^{J}$ (respectively $\widetilde{M}^{J}$) such that for any $w = \sigma \cdot w_J$ with $\sigma\in D_J$ and $w_J \in W_J $,
\begin{center}
$\varphi _J(T_w) = \epsilon _{w_J} m_\sigma^J$, (respectively $\widetilde{\varphi} _J(T_w) = q_{w_J}\cdot \widetilde{m}_\sigma^J$).
\end{center}
Moreover, these maps commute with the standard involution $\bar{\ }$.

\begin{Pro}[\cite{v2}, Theorem 2.1]There is a unique map $\theta_J (m_\sigma^J)=\epsilon_\sigma q_\sigma\overline{\widetilde{m}_\sigma^J}$ from $M^J$ to $\widetilde{M}^{J}$ such that
\begin{center}
$\theta_J (m_e^J)=\widetilde{m}_e^J$ and $\theta_J (h\cdot m_\sigma^J)=\Phi (h)\cdot \theta_J (m_\sigma^J)$,
\end{center}
for any $h\in \mathscr{H} $ and $m_\sigma^J\in M^J$. Furthermore, it has the following properties :

$\left ( 1 \right )$ The following diagram is commutative:
\begin{center}
$\xymatrix{
\mathscr{H} \ar[rr]^{\varphi _J}\ar[d]_{\Phi } &   & M^J \ar[d]^{\theta_J } \\
\mathscr{H} \ar[rr]^{\widetilde{\varphi} _J}   &   & \widetilde{M}^{J} }$
\end{center}

$\left ( 2 \right )$ $\theta_J $ commutes with the involutions $\bar{\ }$ on $M^J$ and $\widetilde{M}^{J}$.

$\left ( 3 \right )$ $\theta_J $ is one-to-one onto and the inverse $\eta_J $ (of $\theta_J $) satisfies properties of $\theta_J $.
\end{Pro}

\subsection{The Hecke modules with respect to $W$-graph ideal}

In this subsection, we recall some constructions of $W$-graph ideal with some modifications. The equal parameter case was introduced by Howlett \cite{r1} and the weighted case was introduced by Yin \cite{y1}.

\begin{Def}[\cite{r1}, Definition 2.2]For any $u,w\in W$, if there are some $s_i\in S$ such that
\begin{center}
$w=s_ks_{k-1}\cdots s_1u$ \ and \ $\ell(s_is_{i-1}\cdots s_1u)=\ell(u)+i,$
\end{center}
for all $ 0\leqslant i\leqslant k$, then we denote $u\leqslant _Lw$ and call $\leqslant _L$ the left weak order on $W$. We say that $u$ is a \emph{\textbf{suffix}} of $w$.
\end{Def}

\begin{Def}[\cite{y1}, Definition 1.1, modified] A $W$-graph for $\mathscr{H}$ consists of the following data:

$\left ( 1 \right )$ A vertex set $\Lambda $ together with a map $I$ which assigns a subset $I(x)\subseteq S$ to each $x\in\Lambda $;

$\left ( 2 \right )$ For each $s\in S$ with $L(s)=\bm{0}$, there is an edge: $x\rightarrow sx$;

$\left ( 3 \right )$ For each $s\in S$ with $L(s)>\bm{0}$, there is a collection of edges such that
\begin{center}
$\left \{ \mu _{x,y}^{s}\in \mathbb{Z}[\Gamma  ]\mid  x,y\in \Lambda , s\in I(x), s\notin I(y) \right \}$ and $\overline{\mu _{x,y}^{s}}=\mu _{x,y}^{s}$.
\end{center}
Furthermore, let $[\Lambda ]_{\mathbb{Z}[\Gamma  ]}$ be a free $\mathbb{Z}[\Gamma  ]$-module with basis $\left \{ b_{y}\mid y\in \Lambda  \right \}$, then we require that the assignment $T_{s}\rightarrow \tau _{s}$ defines a representation of $\mathscr{H}$ by following multiplication.
  \begin{center}
    $\tau_{s}(b_{y})=
     \left\{\begin{aligned}
     &{b_{sy}}                                                                    & \text{if}\ L(s)&=\bm{0};\\
     &{-b_{y}}                                                                    & \text{if}\ L(s)&>\bm{0},s\in I(y);\\
     &{q_{s}b_{y}+q_{s}^{1/2}\sum\limits_{x\in \Lambda ,s\in I(x)}^{ }\mu _{x,y}^{s}b_{x}} & \text{if}\ L(s)&>\bm{0},s\notin I(y).
     \end{aligned}\right.$
  \end{center}
\end{Def}

\begin{Def} [\cite{r1}, Section 5]For any $w\in W$, let $E$ be a subset of $W$ such that $w$ is a suffix of an element of $E$ that is itself in $E$, then we call $E$ is an \text{\bf{ideal}} in the poset $(W,\leqslant _L)$.
\end{Def}

\begin{Def}[\cite{r1}, Definition 2.5] If $E\subseteq W$, let
  \begin{center}
$Pos\left (E\right )=\left \{ s\in S\mid \ell(xs)> \ell (x)\ \text{for all}\ x\in E \right \} $.
  \end{center}
\end{Def}
If $E$ is an ideal of $W$, then $Pos(E) =S/E=\left \{ s\in S|s\notin E \right \}$ and $E\subseteq D_J$, where $J$ is an arbitrary subset of $Pos(E)$. We shall denote by $E_{J}$ for the ideal $E$ with reference to $J$ and the reason is clear from the following constructions. For each $y\in E_{J}$ we define the following subsets of $S$:
   \begin{center}
   $\begin{aligned}
  SD\left ( y \right )&=\left \{s\in S \mid \ell(sy)<\ell (y)\right \},\\
  SA\left ( y \right )&=\left \{s\in S \mid \ell(sy)>\ell (y)\ \text{and}\ sy\in E_{J} \right \},\\
  WD\left ( y \right )&=\left \{s\in S \mid \ell(sy)>\ell (y)\ \text{and}\ sy\notin D_{J} \right \},\\
  WA\left ( y \right )&=\left \{s\in S \mid \ell(sy)>\ell (y)\ \text{and}\ sy\in D_{J}/E_{J} \right \}.
    \end{aligned}$
  \end{center}

\begin{Pro}[\cite{y1}, Definition 2.4, modified] Let $E_{J}$ be an ideal of $W$, it is said to be a $W$-graph ideal with respect to $J$ and $L$ if the following hypotheses are satisfied.

$\left ( 1 \right )$ There exists a $\mathbb{Z}[\Gamma ]$-free $\mathscr{H}$-module $M(E_J,L)$ possessing a basis $\left \{ \Gamma _{y} \mid y\in E_{J}\right \}$ such that
  \begin{center}
   $T_{s}\Gamma _{y}=
      \left\{\begin{aligned}
      &q_s\Gamma _{sy}+\left ( q_s-1 \right )\Gamma _{y}            & \text{if}\ s&\in SD(y),\\
      &\Gamma _{sy}                                                     & \text{if}\ s&\in SA(y),\\
      &-\Gamma _{y}                                                     & \text{if}\ s&\in WD(y),\\
      &q_s\Gamma _{y}-\sum_{z<sy,z\in E_J}^{}r_{z,y}^{s}\Gamma _{z}  & \text{if}\ s&\in WA(y),
      \end{aligned}\right.$
  \end{center}
for some polynomials $r_{z,y}^{s}\in q_s\mathbb{Z}[\Gamma ]$.

$\left ( 2 \right )$ The $\mathscr{H}$-module $M(E_J,L)$ admits a $\mathbb{Z}[\Gamma ]$-semilinear involution $\bar{\ } $ satisfying
\begin{center}
$\overline{\Gamma _{e}}=\Gamma _{e} $ and $\overline{T_{w}\Gamma _{y}}=\overline{T_{w}}\ \overline{\Gamma _{y}}$,
\end{center}
for all $T_{w}\in \mathscr{H}$ and $\Gamma _{y}\in M(E_J,L)$.
\end{Pro}

There is a similarly definition for Hecke module $\widetilde{M}(E_J,L)$ (see \cite{y2}) and we call both of them the weighted $W$-graph ideal modules. we also have duality of them as follows.

\begin{Pro}[\cite{y2}, Theorem 3.1, modified] There is a unique map $\delta(\Gamma_y)=\epsilon_y q_y\overline{\widetilde{\Gamma}_y} \ $ from $M(E_J, L)$ to $\widetilde{M}(E_J, L)$ such that
 \begin{center}
 $\delta \left ( \Gamma _{e} \right )=\widetilde{\Gamma _{e}}$ and $\delta  \left ( T_{w}\Gamma _{y} \right )=\Phi \left ( T_{w} \right )\delta  \left (\Gamma _{y} \right )$,
 \end{center}
for all $T_{w}\in \mathscr{H}$, $\Gamma _{y}\in M(E_J, L)$. It has the following properties:

$\left ( 1 \right )$ $\delta $ commutes with the involution on $M(E_J, L)$ and $\widetilde{M}(E_J, L)$.

$\left ( 2 \right )$ $\delta  $ is a bijection and the inverse $\rho  $ of $\delta  $ satisfies properties of $\eta $.
  \end{Pro}

\section{The relation between $\mathscr{H}$ and $M(E_J,L)$}

Let $E_J$ be a $W$-graph ideal. For brevity, we denote
\begin{center}
$K=Pos(E_J)$ and $\hat{J}=K/J$.
\end{center}
we define
\begin{center}
$D_K=\left \{ w\in W\mid \ell (ws)> \ell (w)\, \text{for all} \ s\in K\right \}$
\end{center}
is the set of minimal coset representative of $W_K$ and $F_J=W_{\hat{J}}$ is the subgroup of $W$ generated by the set $\hat{J}$. It is soon to prove that $E_J \subseteq D_K\subseteq D_J$. Particularly, if $\hat{J}=\o $, then $F_J=\o $ and $D_K=D_J$. We consider this case first.

\subsection{The mapping from $M^{K}$ to $M(E_J,L)$ }

In this subsection, we consider the case of $J=K$ and we denote by $D_K$ for $D_J$. There are two Hecke modules on $D_K$ as well and we denoted by $M^{K}$ and $\widetilde{M}^{K}$. All properties of these modules are similar to $M^{J}$ and $\widetilde{M}^{J}$, we do not give details.

We classify the elements in $D_K$ such that we find the relation between $D_K$ and $E_J$. With the above notations, define
\begin{center}
$\begin{aligned}
&D_K^1=\left \{ \alpha\in D_K \mid \exists\  y\in E_J,\ y\leqslant _L\alpha\right \},\\
&D_K^2=\left \{ \alpha\in D_K \mid \nexists\  y\in E_J,\ y\leqslant _L\alpha\right \},\\
&\overline{E_J}=\left \{ x\in W \mid \alpha=x\cdot y\in D_K^1,\ y\in E_J\right \}.
\end{aligned}$
\end{center}
It is obvious that $E_J\subseteq D_K^1$ and $D_K=D_K^1\cup D_K^2$.

\begin{Theo} For any $\alpha \in D_K^1$, there are unique $x\in \overline{E_J}$ and $y_{max}\in E_J$ such that

$\left ( 1 \right )$ For all $ y\leqslant _L\alpha$, $y\leqslant _L y_{max}$.

$\left ( 2 \right )$ $\alpha = x \cdot y_{max}$ and $\ell(\alpha ) =\ell(x)+\ell(y_{max})$.
\end{Theo}
\begin{proof}This is immediate from the definitions of $E_J$ and $D_K^1$.
\end{proof}

In this paper, we assume throughout that for all $\alpha=x\cdot y\in D_K^1$,
\begin{center}
$y=y_{max}$.
\end{center}

For any $\alpha\in D_K$, define a map $\lambda_J$ from $M^{K}$ to $M(E_J,L)$ by
\begin{center}
    $\lambda_J(m_\alpha^K)=
    \left\{\begin{aligned}
      &q_x\Gamma_y       \ \ \   &\text{Èô}\ &\alpha=x\cdot y\in D_K^1,\\
      &q_\alpha\Gamma_e  \ \ \    &\text{Èô}\ &\alpha\in D_K^2.
    \end{aligned}\right.$
\end{center}
Extend $\lambda_J$ to the whole of $M^{K}$ by linearity, we have:

\begin{Theo} For any $s\in S, \alpha\in D_K$,
\begin{center}
    $\lambda_J(T_s\cdot m_\alpha^K) =
    \left\{\begin{aligned}
      &T_s\cdot \lambda_J( m_\alpha^K)       &\text{if}&\ \alpha \in E_J,  s\in S/ WA(\alpha), \\
      &q_s\cdot \lambda_J( m_\alpha^K)       &\text{if}&\ \alpha \in E_J,  s\in WA(\alpha), s\alpha\in D_K^1,\\
      &-\lambda_J( m_\alpha^K)               &\text{if}&\ \alpha \in E_J,  s\in WA(\alpha), s\alpha\notin D_K,\\
      &q_s\cdot \lambda_J( m_\alpha^K)       &\text{if}&\ \alpha \notin E_J, s\in D_{K, \alpha}^{-}\cup D_{K, \alpha}^{+}, \\
      &-\lambda_J( m_\alpha^K)             &\text{if}&\ \alpha \notin E_J,  s\in D_{K, \alpha}^{0}.
    \end{aligned}\right.$
\end{center}
\end{Theo}

\begin{proof}If $\alpha \in E_J$, the statements can be proved since $\lambda_J(m_y^K)=\Gamma_y$ for any $y\in E_J$. If $\alpha \notin E_J$,  the statements can be proved by the following classification.

$\left ( i \right )$ $s\in D_{K, \alpha}^{-}$. If $\alpha=x\cdot y\in D_K^1/E_J, x\in \overline{E_J}, y\in E_J$,
\begin{center}
$\lambda_J(T_s\cdot m_\alpha^K)=\lambda_J(q_sm_{s\alpha }^{K}+(q_s-1)m_{\alpha }^{K})=q_s\cdot q_{sx}\Gamma_y+(q_s-1)\cdot q_x\Gamma_y=q_s\cdot q_x\Gamma_y.$
\end{center}
If $\alpha\in D_K^2$, $\lambda_J(T_s\cdot m_\alpha^K)=\lambda_J(q_sm_{s\alpha }^{K}+(q_s-1)m_{\alpha }^{K})=q_s\cdot q_{s\alpha}\Gamma_e+(q_s-1)\cdot q_\alpha\Gamma_e=q_s\cdot q_\alpha\Gamma_e.$

$\left ( ii \right )$ $s\in D_{K, \alpha}^{+}$. If $\alpha=x\cdot y\in D_K^1/E_J, x\in \overline{E_J}, y\in E_J$,
\begin{center}
$\lambda_J(T_s\cdot m_\alpha^K)=\lambda_J(m_{s\alpha }^{K})=q_{sx}\Gamma_y=q_s\cdot q_x\Gamma_y.$
\end{center}
If $\alpha\in D_K^2$, $\lambda_J(T_s\cdot m_\alpha^K)=\lambda_J(m_{s\alpha }^{K})=q_{s\alpha}\Gamma_e=q_s\cdot q_\alpha\Gamma_e.$

$\left ( iii \right )$ $s\in D_{K, \alpha}^{0}$. If $\alpha=x\cdot y\in D_K^1/E_J, x\in \overline{E_J}, y\in E_J$,
\begin{center}
$\lambda_J(T_s\cdot m_\alpha^K)=\lambda_J(-m_{\alpha }^{K})=-q_x\Gamma_y.$
\end{center}
If $\alpha\in D_K^2$, $\lambda_J(T_s\cdot m_\alpha^K)=\lambda_J(-m_{\alpha }^{K})=-q_\alpha\Gamma_e.$
\end{proof}

\begin{Cor} For any $\alpha\in D_K$, let $\alpha=s_1s_2\cdots s_k \in D_K$ be a reduced word.

(1) If $\alpha\notin E_J$, then there are some $s_{i_1}, s_{i_2}, \cdots , s_{i_m}\in D_{K, \alpha}^{-}\cup D_{K, \alpha}^{+}$, $s_{j_1}, s_{j_2},\cdots, s_{j_n}\in D_{K, \alpha}^{0}$ such that $m+n=k$ and
\begin{center}
$\lambda_J(m_\alpha^K)=(-1)^n q^\mu  \Gamma_e$,
\end{center}
where $\mu=L(s_{i_1})+L(s_{i_2})+\cdots+L(s_{i_m})$.

(2) If $\alpha\in E_J$, then there are some $s_{i_1}, s_{i_2}, \cdots , s_{i_m}\in  S/ WA(\alpha)$, $s_{j_1}, s_{j_2},\cdots, s_{j_n}\in \left \{ s\in WA(\alpha)\mid s\alpha\in D_K^1 \right \}$,  $s_{l_1}, s_{l_2},\cdots, s_{l_t}\in \left \{ s\in WA(\alpha)\mid s\alpha\notin D_K \right \}$ such that $m+n+t=k$ and
\begin{center}
$\lambda_J(m_\alpha^K)=(-1)^t q^\mu T_{s_{i_1} s_{i_2} \cdots s_{i_m}} \Gamma_e$,
\end{center}
where $\mu=L(s_{j_1})+L(s_{j_2})+\cdots+L(s_{j_n})$.
\end{Cor}

\begin{Cor}
$\lambda_J$ commutes with the involution $\bar{\ }$ on $M^{K}$ and $M(E_J,L)$.
\end{Cor}
\begin{proof}For any $s\in S, \alpha\in D_K$, it can be proved that
\begin{center}
    $\lambda_J(\overline{T_s}\cdot m_\alpha^K) =
    \left\{\begin{aligned}
      &\overline{T_s}\cdot \lambda_J( m_\alpha^K)       &\text{if}&\ \alpha \in E_J,  s\in S/ WA(\alpha), \\
      &q_s^{-1}\cdot \lambda_J( m_\alpha^K)       &\text{if}&\ \alpha \in E_J,  s\in WA(\alpha), s\alpha\in D_K^1,\\
      &-\lambda_J( m_\alpha^K)               &\text{if}&\ \alpha \in E_J,  s\in WA(\alpha), s\alpha\notin D_K,\\
      &q_s^{-1}\cdot \lambda_J( m_\alpha^K)       &\text{if}&\ \alpha \notin E_J, s\in D_{K, \alpha}^{-}\cup D_{K, \alpha}^{+}, \\
      &-\lambda_J( m_\alpha^K)             &\text{if}&\ \alpha \notin E_J,  s\in D_{K, \alpha}^{0}.
    \end{aligned}\right.$
\end{center}
then the result is immediate from Corollary 2.3.
\end{proof}

\begin{Rem}For any $\alpha\in D_K$, there is a map $\widetilde{\lambda}_J$ from $\widetilde{M}^{K}$ to $\widetilde{M}(E_J,L)$ defined by
\begin{center}
    $\widetilde{\lambda}_J(\widetilde{m}_\alpha^K)=
    \left\{\begin{aligned}
      &\epsilon_x\widetilde{\Gamma_y}       \ \ \   &\text{if}\ &\alpha=x\cdot y\in D_K^1,\\
      &\epsilon_\alpha\widetilde{\Gamma_e}  \ \ \    &\text{if}\ &\alpha\in D_K^2.
    \end{aligned}\right.$
\end{center}
and the properties of $\widetilde{\lambda}_J$ is very similar to $\lambda_J$.
\end{Rem}

\begin{Theo}The following diagram is commutative:
\begin{center}
$\xymatrix{
M^K  \ar[rr]^{\lambda_J }\ar[d]_{\theta_K } &               & M(E_J,L) \ar[d]^{\delta} \\
\widetilde{M}^{K} \ar[rr]^{\widetilde{\lambda}_J}   &   & \widetilde{M}(E_J,L)}$
\end{center}
\end{Theo}

\begin{proof}
For any $\alpha \in D_K$, we have
\begin{center}
    $\delta\circ \lambda_J \left ( m_\alpha^K \right )=
    \left\{\begin{aligned}
      &\delta\left ( q_x\Gamma_y\right )=\epsilon _yq_\alpha\overline{\widetilde{\Gamma}_y}    \ \ \   &\text{Èô}\ &\alpha=x\cdot y\in D_K^1,\\
      &\delta\left ( q_\alpha\Gamma_e\right )=q_\alpha\overline{\widetilde{\Gamma}_e}  \ \ \    &\text{Èô}\ &\alpha\in D_K^2.
    \end{aligned}\right.$
\end{center}
and
\begin{center}
    $\widetilde{\lambda }_J\circ \theta_K\left ( m_\alpha^K\right )=
    \left\{\begin{aligned}
      &\widetilde{\lambda }_J\left (\epsilon _\alpha q_\alpha\overline{\widetilde{m}_\alpha^K}\right )=\epsilon _x\epsilon _\alpha q_\alpha\overline{\widetilde{\Gamma}_y }   \ \ \   &\text{Èô}\ &\alpha=x\cdot y\in D_K^1,\\
      &\widetilde{\lambda }_J\left (\epsilon _\alpha q_\alpha\overline{\widetilde{m}_\alpha^K}\right )= q_\alpha\overline{\widetilde{\Gamma}_e}  \ \ \    &\text{Èô}\ &\alpha\in D_K^2.
    \end{aligned}\right.$
\end{center}
Clearly, these two expressions are equal.
\end{proof}

\subsection{The mapping from $M^{J}$ to $M^{K}$ }

In this subsection, we consider the case of $J\subset K$. In order to find the relation between $M^{J}$ and $M(E_J,L)$, we construct the mapping from $M^{J}$ to $M^{K}$.

\begin{Theo} With the above notations, let $J\subset K$. Then we have

$\left ( 1 \right )$ $D_K\subset D_J$, $F_J\subset D_J$ and $D_K\cap F_J=\left \{ e \right \}$.

$\left ( 2 \right )$ Every $\sigma\in D_J$ has a unique factorization $\sigma = \alpha \cdot z$ such that $\alpha \in D_K$, $z\in F_J$ and
\begin{center}
$\ell(\sigma) =\ell(\alpha )+\ell(z)$.
\end{center}
\end{Theo}

\begin{proof} We first prove part $\left ( 1 \right )$, $D_K\subset D_J$ and $F_J\subset D_J$ are obvious from $J\subset K$ and $\hat{J}\subset S/J$. By the definition of $D_K$ and $F_J$, we easy to check that $D_K\cap F_J=\left \{ e \right \}$.

The proof of part $\left ( 2 \right )$ is similar to \cite{a1}, Proposition 2.4.4, we omit the details.
\end{proof}

For any $\sigma\in D_J$ and $\sigma = \alpha \cdot z$ with $\alpha \in D_K$, $z\in F_J$. Define a map $\lambda_K$ from $M^{J}$ to $M^{K}$ by
\begin{center}
$\lambda_K(m_\sigma^J) = \epsilon_zm_\alpha ^K$.
\end{center}
Extend $\lambda_K$ to the whole of $M^{J}$ by linearity. Then we have:

\begin{Theo} $\lambda_K$ has $\mathscr{H}$-linearity, i.e. for $s\in S, \sigma\in D_J$,
\begin{center}
$\lambda_K(T_s\cdot m_\sigma^J) =T_s\cdot \lambda_K(m_\sigma^J)$.
\end{center}
\end{Theo}

\begin{proof}This is done by 'case by case' considerations. The most interesting case occurs when $s\sigma = s\alpha\cdot z=\alpha s'\cdot z>\sigma $, where $\alpha\in D_K, z\in F_J, s'\in K$. In this case,
\begin{center}
$T_s\cdot \lambda_K(m_\sigma^J)=\epsilon _zT_s\cdot m_\alpha^K=-\epsilon _zm_\alpha^K$.
\end{center}
On the other hand, for any $s'\in \hat{J}$, if $s'z<z$, then $s\in D_{J, \sigma}^{-}$,
\begin{center}
$\lambda_K(T_sm_\sigma^J)=\lambda_K(q_sm_{s\sigma}^J+\left ( q_s-1 \right )m_\sigma^J)=-\epsilon _zq_sm_\alpha^K+\epsilon _z(q_s-1)m_\alpha^K=-\epsilon _zm_\alpha^K.$
\end{center}
If $s'z>z$, then $s\in D_{J, \sigma}^{+}$,
\begin{center}
$\lambda_K(T_sm_\sigma^J)=\lambda_K(m_{s\sigma}^J)=-\epsilon _zm_\alpha^K$.
\end{center}
For $s'\in J$, we have $s'z>z$ and $s'z\notin F_J$, then $s\in D_{J, \sigma}^{0}$,
\begin{center}
$\lambda_K(T_sm_\sigma^J)=\lambda_K(-m_{\sigma}^J)=-\epsilon _zm_\alpha^K$.
\end{center}
This is exactly the case. Other cases are simpler than this and we omit the details.
\end{proof}

\begin{Cor} For any $\sigma\in D_J$ and $\sigma = \alpha \cdot z$ with $\alpha\in D_K, z\in F_J$,
\begin{center}
$\lambda_K(m_\sigma^J)=T_\sigma\cdot m_e^K$.
\end{center}
\end{Cor}

\begin{proof} By $m_\sigma^J=T_\sigma\cdot m_e^J$, we get
\begin{center}
$\lambda_K(m_\sigma^J) =\lambda_K(T_\sigma\cdot m_e^J) =T_{\sigma}\cdot \lambda_K(m_e^J)=T_\sigma\cdot m_e^K$.
\end{center}
\end{proof}

\begin{Cor}
$\lambda_K$ commutes with the involution $\bar{\ }$ on $M^{J}$ and $M^{K}$.
\end{Cor}

\begin{proof}This is immediate from Corollary 2.9.
\end{proof}

\begin{Rem}There is a map $\widetilde{\lambda}_K$ from $\widetilde{M}^{J}$ to $\widetilde{M}^{K}$ defined by
\begin{center}
$\widetilde{\lambda}_K(\widetilde{m}_\sigma^J) =  q_z \widetilde{m}_\alpha^K$.
\end{center}
and the properties of $\widetilde{\lambda}_K$ is very similar to $\lambda_K$.
\end{Rem}

\begin{Theo}The following diagram is commutative:
\begin{center}
$\xymatrix{
M^J  \ar[rr]^{\lambda_K }\ar[d]_{\theta_J } &               & M^K \ar[d]^{\theta_K} \\
\widetilde{M}^{J} \ar[rr]^{\widetilde{\lambda}_K}   &   & \widetilde{M}^K}$
\end{center}
\end{Theo}

\begin{proof}
For $\sigma =\alpha\cdot z$ with $\alpha \in D_K$ and $z\in F_J $, we get
\begin{center}
$\theta_K \circ \lambda_K \left ( m_\sigma^J \right )=\theta_K\left ( \epsilon _zm_\alpha^K\right )=\epsilon _\sigma q_\alpha\overline{\widetilde{m }_\alpha^K}$,
\end{center}
and
\begin{center}
$\widetilde{\lambda }_K\circ \theta_J\left ( m_\sigma^J \right )=\widetilde{\lambda }_K\left ( \epsilon _\sigma q_\sigma\overline{\widetilde{m}_\sigma^J}\right )=\epsilon _\sigma q_\sigma q_z^{-1}\overline{\widetilde{m }_\alpha^K}$
\end{center}
Clearly, these two expressions are equal.
\end{proof}

\subsection{The mapping from $\mathscr{H}$ to $M(E_J,L)$}

Finally, we can give the relation between $\mathscr{H}$ and $M(E_J,L)$.

\begin{Theo} For any $w\in W$, there is a unique factorization
\begin{center}
$w= \alpha \cdot z\cdot w_J$,
\end{center}
where $\alpha\in D_K, z\in F_J, w_J\in W_J$. Then we can defined a map $\nu $ from $\mathscr{H} $ to $M(E_J,L)$ by
\begin{center}
    $\nu (T_w)=
    \left\{\begin{aligned}
      &\epsilon_z\epsilon _{w_J} q_x\Gamma_y       \   &\text{if}\ &\alpha=x\cdot y\in D_K^1,\\
      &\epsilon_z \epsilon _{w_J} q_\alpha\Gamma_e  \ \ \    &\text{if}\ &\alpha\in D_K^2.
    \end{aligned}\right.$
\end{center}
such that

$\left ( 1 \right )$ $\nu$ commutes with the involution $\bar{\ }$.

$\left ( 2 \right )$ The following diagram is commutative:
\begin{center}
$\xymatrix{
\mathscr{H} \ar[rr]^{\varphi _J}\ar[d]_{\Phi }  &   & M^J  \ar[rr]^{\lambda_K}\ar[d]_{\theta_J }          &   &
 M^K \ar[rr]^{\lambda_J} \ar[d]^{\theta_K}      &   & M(E_J,L) \ar[d]^{\delta} \\
\mathscr{H} \ar[rr]^{\widetilde{\varphi} _J}    &   & \widetilde{M}^{J} \ar[rr]^{\widetilde{\lambda}_K}   &   & \widetilde{M}^{K}\ar[rr]^{\widetilde{\lambda}_J}&   & \widetilde{M}(E_J,L)}$
\end{center}
Note that if $J=K$, then $M^{J, u_s}=M^{K, u_s}$ and $\lambda _K=\widetilde{\lambda}_K=Id$.
\end{Theo}
\begin{proof}
This is proved by Proposition 1.1, Proposition 1.7, Theorem 2.6 and Theorem 2.12.
\end{proof}

\section{An isomorphism of $M^{J}$ onto left-ideal of $\hat{\mathscr{H}}$}

In this section, we show that the parabolic Hecke module $M^J$ is isomorphic to a left ideal of the Hecke algebra. The idea in this section is inspired by Lusztig \cite{g3} and the method of proof is very similar to that paper.

Let $\hat{\mathscr{H}}$ be the vector space consisting of all formal sums $\sum\limits_{w\in W}^{}c_wT_w$, where $c_w\in \mathbb{Z}[\Gamma ]$. We can view $\mathscr{H}$ as a subspace of $\hat{\mathscr{H}}$ in an obvious way. The $\mathscr{H}$-module structure on $\mathscr{H}$ extends to a $\mathscr{H}$-module structure on $\hat{\mathscr{H}}$. For any $z\in D_J$, we set
\begin{center}
$Q_z=\sum\limits_{y\in W_J}^{}N_y^zT_{zy}\in \hat{\mathscr{H}}$,
\end{center}
where $N_y^z=\epsilon _yq_y^{-1}$ will be proved in Theorem 3.2. Then $Q_J=\left \langle Q_z \right \rangle_{z\in D_J}$ as a $\mathscr{H}$-submodule generated by
\begin{center}
$Q_e=\sum\limits_{y\in W_J}^{}\epsilon _yq_y^{-1}T_y$.
\end{center}

\begin{Pro}
With the above notations, $Q_J$ is a left-ideal of $\hat{\mathscr{H}}$.
\end{Pro}
\begin{proof}For any $z_1\neq z_2 \in D_J$, we have
\begin{center}
$Q_{z_1}=N_e^{z_1}T_{z_1\cdot e}+N_{s_1}^{z_1}T_{z_1s_1}+N_{s_1s_2}^{z_1}T_{z_1s_1s_2}+\cdots+N_{y}^{z_1}T_{z_1y}+\cdots$

$Q_{z_2}=N_e^{z_2}T_{z_2\cdot e}+N_{s_1}^{z_2}T_{z_2s_1}+N_{s_1s_2}^{z_2}T_{z_2s_1s_2}+\cdots+N_{y}^{z_2}T_{z_2y}+\cdots$
\end{center}
where $s_1, s_2, \cdots, s_i\in J, y\in W_J$. Thus $Q_{z_1}-Q_{z_2}\in Q_J$ is obvious since $Q_J$ is a group.

For any $w\in W$, we can give $wz$ a unique factorization $wz=uv$ such that $u\in D_J$ and $v\in W_J$, then
\begin{center}
$T_wQ_z=\sum\limits_{y\in W_J}^{}N_y^zT_wT_{xy}=\sum\limits_{vy\in W_J}^{}\tilde{N}_{vy}^zT_{uvy}\in Q_J$,
\end{center}
this is proved that $Q_J$ is a left-ideal of $\hat{\mathscr{H}}$.
\end{proof}

\begin{Theo}There exists a unique $\mathscr{H}$-linear map $\mu:M^J \rightarrow Q_J$ such that
\begin{center}
$\mu(m_e^J)=\sum\limits_{y\in W_J}^{}\epsilon _yq_y^{-1}T_y$.
\end{center}
Moreover, $\mu$ is an isomorphism of $M^J$ onto the $\mathscr{H}$-submodule of $\hat{\mathscr{H}}$ generated by $\mu(m_e^J)$.
\end{Theo}

\begin{proof}For any $z\in D_J$, giving a $\mathbb{Z}[\Gamma ]$-linear map $\mu:M^J \rightarrow Q_J$ by
\begin{center}
$\mu(m_z^J)=\sum\limits_{y\in W_J}^{} N_y^zT_{zy}$,
\end{center}
where $\left \{ N_y^z \mid (y,z)\in W_J\times D_J \right \}\subset \mathbb{Z}[\Gamma ]$. It is obvious that $\mu$ is surjective. For $z_1\neq z_2 \in D_J$, we have
\begin{center}
$z_1W_J\cap z_2W_J=\{e\}$ and $m_{z_1}^J\neq m_{z_2}^J\in M^J$,
\end{center}
then $\mu(m_{z_1}^J)=Q_{z_1}\neq Q_{z_2}=\mu(m_{z_2}^J)$. So the injectivity of $\mu$ is proved.

We will prove that $\mu$ is a $\mathscr{H}$-linear isomorphism as follows. For all $z\in D_J, s\in S$, the $\mathscr{H}$-linearity of $\mu$ is equivalent to the following equation:
\begin{center}
$\mu (T_sm_z^J)=T_s\cdot\mu(m_z^J)$.
\end{center}
This is equivalent to
\begin{center}
$\begin{aligned}
T_{s}Q_z=&q_sQ_{sz}+(q_s-1)Q_z            & \text{if}\ s\in D_{J, z}^{-},\\
T_{s}Q_z=&Q_{sz}                                           & \text{if}\ s\in D_{J, z}^{+},\\
T_{s}Q_z=&-Q_{z}                                           & \text{if}\ s\in D_{J, z}^{0},
\end{aligned}$
\end{center}
Since $Q_z=\sum\limits_{y\in W_J}^{}N_y^zT_{zy}=\sum\limits_{y\in W_J}^{}N_y^zT_zT_y$, we see that giving a $\mathscr{H}$-linear map $\mu$ is the same as giving a family of elements $\left \{ N_y^z \mid (y,z)\in W_J\times D_J \right \}$ in $\mathbb{Z}[\Gamma ]$ such that the following equations are satisfied for any $z\in D_J, s\in S$,
\begin{center}
$\begin{aligned}
\sum\limits_{y\in W_J}^{}N_y^zT_{s}T_zT_y=&\sum\limits_{y\in W_J}^{}(q_sN_{y}^{sz}T_{sz}+(q_s-1)N_y^zT_z)T_y  & \text{if}\ s\in D_{J, z}^{-},\\
\sum\limits_{y\in W_J}^{}N_y^zT_{s}T_zT_y=&\sum\limits_{y\in W_J}^{}N_y^{sz}T_{sz}T_y        & \text{if}\ s\in D_{J, z}^{+},\\
\sum\limits_{y\in W_J}^{}N_y^zT_{s}T_zT_y=&-\sum\limits_{y\in W_J}^{}N_y^zT_zT_y           & \text{if}\ s\in D_{J, z}^{0},
\end{aligned}$
\end{center}
Here we replace
\begin{center}
$\begin{aligned}
\sum\limits_{y\in W_J}^{}N_y^zT_{s}T_zT_y=&\sum\limits_{y\in W_J}^{}N_y^z(q_sT_{sz}+(q_s-1)T_z)T_y  & \text{if}\ &s\in D_{J, z}^{-},\\
\sum\limits_{y\in W_J}^{}N_y^zT_{s}T_zT_y=&\sum\limits_{y\in W_J}^{}N_y^{z}T_{sz}T_y        & \text{if}\ &s\in D_{J, z}^{+},
\end{aligned}$
\end{center}
Especially, if $s\in D_{J, z}^{0}$, there is a $t\in J$ such that $sz=zt$, then
\begin{center}
$\begin{aligned}
\sum\limits_{y\in W_J}^{}N_y^zT_{s}T_zT_y&=\sum\limits_{y\in W_J}^{}N_y^zT_zT_tT_y  \\
&=\sum\limits_{y\in W_J,ty>y}^{}N_y^zT_zT_{ty} +\sum\limits_{y\in W_J,ty<y}^{}N_y^zT_z(q_sT_{ty}+(q_s-1)T_y) \\
&=\sum\limits_{y\in W_J,ty<y}^{}(N_{ty}^z+(q_s-1)N_y^z)T_zT_{y} +\sum\limits_{y\in W_J,ty>y}^{}q_sN_{ty}^zT_zT_{y}
\end{aligned}$
\end{center}
We see that $\left \{ N_y^z \mid (y,z)\in W_J\times D_J \right \}$ must satisfy the following set of equations:
\begin{center}
$(1)\left\{\begin{aligned}
N_y^{z}&=N_y^{sz}               & \text{if}\ &s\in D_{J, z}^{-},         \\
N_y^{z}&=N_y^{sz}      & \text{if}\ &s\in D_{J, z}^{+},   \\
N_y^{z}&=-q_s^{-1}N_{ty}^z              & \text{if}\ &s\in D_{J, z}^{0},ty<y,\\
N_y^{z}&=-q_sN_{ty}^z              & \text{if}\ &s\in D_{J, z}^{0},ty>y.
\end{aligned}\right.$
\end{center}

On the other hand, by the definition of involution on $M^J$, we have $T_zm_e^J=m_z^J$ for any $z\in D_J$. Hence for any $x\in W$ with $x=zy, z\in D_J, y\in W_J$, we can write uniquely
\begin{center}
$T_xm_e^J=T_zT_ym_e^J=L_y^zm_z^J$,
\end{center}
where $L_y^z\in \mathbb{Z}[\Gamma ]$. For any $x\in W, s\in S$, we have
\begin{center}
$T_sT_zT_ym_e^J=L_y^zT_sm_z^J$.
\end{center}

For the case $s\in D_{J, z}^{-}$, we have
\begin{center}
$q_sL_y^{sz}m_{sz}^J+(q_s-1)L_y^zm_z^J=q_sL_y^{z}m_{sz}^J+(q_s-1)L_y^zm_z^J$.
\end{center}

For the case $s\in D_{J, z}^{+}$, we have $L_y^{sz}m_{sz}^J=L_y^{z}m_{sz}^J$.

For the case $s\in D_{J, z}^{0}$ and $ty<y$, we have
\begin{center}
$q_sL_{ty}^{z}m_{z}^J+(q_s-1)L_y^zm_z^J=-L_y^{z}m_{z}^J$.
\end{center}

For the case $s\in D_{J, z}^{0}$ and $ty>y$, we have $L_{ty}^{z}m_{z}^J=-L_y^{z}m_{z}^J$.

Summing this four cases up, we have
\begin{center}
$(2)\left\{\begin{aligned}
L_y^{z}&=L_y^{sz}               & \text{if}\ &s\in D_{J, z}^{-},         \\
L_y^{z}&=L_y^{sz}       & \text{if}\ &s\in D_{J, z}^{+},   \\
L_y^{z}&=-L_{ty}^{z}              & \text{if}\ &s\in D_{J, z}^{0},ty<y,\\
L_y^{z}&=-L_{ty}^{z}              & \text{if}\ &s\in D_{J, z}^{0},ty>y.
\end{aligned}\right.$
\end{center}
We replace $L_y^z$ with $\tilde{L}_y^z$ by $\tilde{L}_y^z=q_y^{-1}L_y^z$, then the formulas of $L_y^z$ can be rewritten as follows:
\begin{center}
$(3)\left\{\begin{aligned}
\tilde{L}_y^{z}&=\tilde{L}_y^{sz}               & \text{if}\ &s\in D_{J, z}^{-},         \\
\tilde{L}_y^{z}&=\tilde{L}_y^{sz}       & \text{if}\ &s\in D_{J, z}^{+},   \\
\tilde{L}_y^{z}&=-q_s^{-1}\tilde{L}_{ty}^{z}              & \text{if}\ &s\in D_{J, z}^{0},ty<y,\\
\tilde{L}_y^{z}&=-q_s\tilde{L}_{ty}^{z}              & \text{if}\ &s\in D_{J, z}^{0},ty>y.\\
\end{aligned}\right.$
\end{center}
Comparing with the formulas in (1), we see that the conditions of $\left \{ \tilde{L}_y^z \mid (y,z)\in W_J\times D_J \right \}$ is equivalent to the conditions of $\left \{ N_y^z \mid (y,z)\in W_J\times D_J \right \}$.

Hence there is a unique $\mathscr{H}$-linear map $\mu:M^J \rightarrow Q_J$ such that
\begin{center}
$\mu(m_z^J)=\sum\limits_{y\in W_J}^{} q_y^{-1}L_y^zT_{zy}$,
\end{center}
for any $z\in D_J$. Actually, $L_y^z=\epsilon _{y}$ is easy to find since the definition of $M^J$.
\end{proof}

\section{The relation of the corresponding $R$-polynomials}

As an application of Theorem 2.13, we give the relation between the parabolic $R$-polynomials on $D_J$ and the $R$-polynomials on $E_J$. First, we recall some definitions.

\begin{Def}[\cite{h1}, Definition 3.1]There exists a family of polynomials $\left\{R_{\sigma,\tau }^J\mid \sigma ,\tau \in D_J \right \}$ satisfying the condition
   \begin{center}
   $\overline{m_{\tau}^J}=\sum\limits_{\sigma\in D_J}^{}\epsilon _\sigma\epsilon _\tau q_\tau^{-1}R_{\sigma,\tau}^J m_{\sigma}^J$,
   \end{center}
we call these polynomials the parabolic $R$-polynomials on $D_J$.
\end{Def}

\begin{Def}[\cite{y1}, Corollary 3.2]There exists a family of polynomials $\left\{R_{x,y}\mid x,y \in E_J \right \}$ satisfying the condition
   \begin{center}
  $ \overline{\Gamma _{y}}=\sum\limits_{x\in E_J}^{}\epsilon _{x}\epsilon _{y}q_{y}^{-1}R_{x,y}\Gamma _{x},$
   \end{center}
we call these polynomials the $R$-polynomials on $E_J$.
\end{Def}

Then, we have
\begin{Theo}With the above notations, for any $x, y\in E_J$,
\begin{center}
    $R_{x,y}=
    \left\{\begin{aligned}
      &\sum\limits_{ u\in \overline{E_J}, z\in F_J }^{}\epsilon_uq_uR_{uxz,y}^J      &\text{Èô}\ x\neq e,\\
      &\sum\limits_{ u\in \overline{E_J}, z\in F_J }^{}\epsilon_uq_uR_{uz,y}^J +\sum\limits_{\alpha\in D_K^2, z\in F_J}^{}\epsilon_\alpha q_\alpha R_{\alpha z, y}^J               &\text{Èô}\ x=e.
    \end{aligned}\right.$
  \end{center}
\end{Theo}
\begin{proof}
This is clear from the following two subsections.
\end{proof}

\begin{Rem}The relation between $\widetilde {R}_{\sigma,\tau }^J$ and $R_{x,y}$ as follows£»
\begin{center}
    $\widetilde {R}_{x,y}=
    \left\{\begin{aligned}
      &\sum\limits_{ u\in \overline{E_J}, z\in F_J }^{}\epsilon_zq_z\widetilde {R}_{uxz,y}^J    &\text{Èô}\ x\neq e,\\
      &\sum\limits_{ u\in \overline{E_J}, z\in F_J }^{}\epsilon_zq_z\widetilde {R}_{uz,y}^J +\sum\limits_{\alpha\in D_K^2, z\in F_J}^{}\epsilon_zq_z\widetilde {R}_{\alpha z, y}^J               &\text{Èô}\ x=e.
    \end{aligned}\right.$
  \end{center}
\end{Rem}

\subsection{The relation between $R_{x,y}$ and $R_{\alpha,\beta}^K$ }

\begin{Def}There exists a unique family of polynomials $\left \{ R_{\alpha,\beta}^K\mid \alpha,\beta\in D_K \right \}$ satisfying the condition
   \begin{center}
   $\overline{m_{\beta}^K}=\sum\limits_{\alpha\in D_K}^{}\epsilon _\alpha\epsilon _\beta q_\beta^{-1}R_{\alpha,\beta}^K m_{\alpha}^K$,
   \end{center}
we call these polynomials the parabolic $R$-polynomials on $D_K$.
\end{Def}

\begin{Theo}With the above notations, for any $x, y\in E_J$,
\begin{center}
    $R_{x,y}=
    \left\{\begin{aligned}
      &\sum\limits_{ z\in \overline{E_J} }^{}\epsilon _zq_zR_{zx,y}^K                &\text{Èô}\ x\neq e,\\
      &\sum\limits_{ z\in \overline{E_J} }^{}\epsilon _zq_zR_{z,y}^K +\sum\limits_{\alpha\in D_K^2}^{}\epsilon _\alpha q_\alpha R_{\alpha,y}^K               &\text{Èô}\ x=e.
    \end{aligned}\right.$
  \end{center}
\end{Theo}

\begin{proof} By the definition of $\lambda_J$, we easy to check that $\overline{\Gamma}_{y}=\lambda_J(\overline{m_y^K})$ for any $y\in E_J$, then
 \begin{center}
      $\begin{aligned}\overline{\Gamma}_{y}
      &=\sum\limits_{\alpha\in D_K}^{}\epsilon_\alpha \epsilon_yq_y^{-1}R_{\alpha,y}^K \lambda _J(m_{\alpha}^K)\\
      &=\epsilon_\alpha \epsilon_yq_y^{-1}\left (\sum\limits_{\substack{\alpha=zx\in D_K^1,\\ z\in \overline{E_J}, x\in E_J }}^{}q_zR_{zx,y}^K \Gamma_{x}+\sum\limits_{\alpha\in D_K^2}^{}q_\alpha R_{\alpha,y}^K \Gamma_{e} \right )\\
      &=\sum\limits_{ e\neq x \in E_J }^{}\left (\sum\limits_{ z\in \overline{E_J} }^{}\epsilon_x\epsilon_y\epsilon_z q_y^{-1}q_zR_{zx,y}^K \right)\Gamma_{x}\\
      &\ \ \ \ +\left (\sum\limits_{ z\in \overline{E_J} }^{}\epsilon_y\epsilon_z q_y^{-1}q_zR_{z,y}^K +\sum\limits_{\alpha\in D_K^2}^{}\epsilon_\alpha\epsilon_y q_y^{-1}q_\alpha R_{\alpha,y}^K \right )\Gamma_{e}
      \end{aligned}$
     \end{center}
On the other hand,
\begin{center}
$\overline{\Gamma}_{y}=\sum\limits_{x\in E_J}^{}\epsilon_x\epsilon_y q_y^{-1}R_{x,y}\Gamma_{x}=\sum\limits_{e\neq x\in E_J}^{}\epsilon_x\epsilon_y q_y^{-1}R_{x,y}\Gamma_{x}+\epsilon_y q_y^{-1}R_{e,y}\Gamma_{e}$
\end{center}
Comparing the coefficients of $\Gamma_{x}$ and $\Gamma_{e}$ in the two expressions, we get the result.
\end{proof}

\begin{Rem}There is a relation between $\widetilde {R}_{x,y}$ and $\widetilde {R}_{\alpha,\beta}^K$.
\begin{center}
    $\widetilde {R}_{x,y}=
    \left\{\begin{aligned}
      &\sum\limits_{ z\in \overline{E_J} }^{}\widetilde {R}_{zx,y}^K                &\text{Èô}\ x\neq e,\\
      &\sum\limits_{ z\in \overline{E_J} }^{}\widetilde {R}_{z,y}^K +\sum\limits_{\alpha\in D_K^2}^{}\widetilde {R}_{\alpha,y}^K               &\text{Èô}\ x=e.
    \end{aligned}\right.$
  \end{center}
\end{Rem}

\subsection{The relation between $R_{\alpha, \beta}^K$ and $R_{\sigma,\tau}^J$  }

\begin{Theo}For any $\alpha, \beta\in D_K$,
\begin{center}
    $R_{\alpha,\beta}^K=\sum\limits_{z\in F_J}^{}R_{\alpha z, \beta}^J $.
  \end{center}
\end{Theo}

\begin{proof} By the definition of $\lambda_K$, we easy to check that $\overline{m_{\beta}^K}=\lambda_K(\overline{m_\beta^J})$ for any $\beta\in D_K$, then for $\sigma = \alpha \cdot z$ with $\alpha \in D_K$ and $ z\in F_J$,
\begin{center}
      $\begin{aligned}\overline{m_{\beta}^K}
      &=\sum\limits_{\sigma\in D_J}^{}\epsilon_\sigma\epsilon_\beta q_\beta^{-1}R_{\sigma,\beta}^J \lambda_K(m_{\sigma}^J)\\
      &=\sum\limits_{\substack{\alpha \in D_K,\\ z\in F_J}}^{}\epsilon_\alpha\epsilon_\beta q_\beta^{-1}R_{\alpha z, \beta}^J m_{\alpha}^K\\
      &=\sum\limits_{\alpha \in D_K}^{}\left(\sum\limits_{z\in F_J}^{}\epsilon_\alpha\epsilon_\beta q_\beta^{-1}R_{\alpha z, \beta}^J \right) m_{\alpha}^K
      \end{aligned}$
     \end{center}
On the other hand,
\begin{center}
$\overline{m_{\beta}^K}=\sum\limits_{\alpha\in D_K}^{}\epsilon_\alpha\epsilon_\beta q_\beta^{-1}R_{\alpha,\beta}^K m_{\alpha}^K$
\end{center}
Comparing the coefficients of $m_{\alpha}^K$ in the two expressions, we get the result.
\end{proof}

\begin{Rem}There is a relation between $\widetilde {R}_{\alpha,\beta}^K$ and $\widetilde {R}_{\sigma,\tau }^J$.
\begin{center}
    $\widetilde {R}_{\alpha,\beta}^K=\sum\limits_{z\in F_J}^{}\epsilon_zq_z\widetilde {R}_{\alpha z, \beta}^J $
  \end{center}
\end{Rem}

It is known that Theorem 4.3 is founded by Theorem 4.6 and Theorem 4.8.

\noindent\emph{\textbf{Email:} q.wang@163.sufe.edu.cn}

\noindent\emph{\textbf{Address: }School of Mathematics, Shanghai University of Finance and Economics, Guoding Road No.777, Shanghai, China, 200433.}

\end{spacing}
\end{document}